\documentclass[12pt,reqno]{amsart}
\usepackage[a4paper,margin=2.5cm,top=2.5cm,bottom=2.5cm,centering,headheight=3ex,headsep=4ex,vcentering]{geometry}
\usepackage{amssymb,amsfonts,amsmath,amsthm}
\usepackage{cite}
\usepackage{orcidlink}
\usepackage[cal=cm,scr=euler]{mathalfa}
\usepackage{microtype}
\usepackage{mlmodern}
\numberwithin{equation}{section}
\theoremstyle{plain}
\newtheorem{theorem}{Theorem}
\newtheorem{conjecture}{Conjecture}
\newtheorem{lemma}{Lemma}

\theoremstyle{definition}

\theoremstyle{remark}

\usepackage{xcolor}
\usepackage{hyperref}
\definecolor{ao(english)}{rgb}{0.0, 0.5, 0.0}
\definecolor{pink}{rgb}{1.0, 0.0, 1.0}
\hypersetup{colorlinks=true, linkcolor=ao(english),citecolor=ao(english)}

\title[Two parametric $q$-supercongruences from Andrews's Multi-series transformation]{Two parametric $q$-supercongruences from Andrews' multi-series transformation}

\author[Liton Karmakar]{Liton Karmakar\,\orcidlink{0009-0002-9532-1677}}
\address[Liton Karmakar]{Department of Mathematics, National Institute of Technology Silchar, Assam 788010, India}
\email{litonofficial8638@gmail.com}


\begin{document}

\begin{abstract}
Recently, Guo and Guo (Results Math $76:37, ~2021$) obtained curious $q$-analogues of the following two supercongruences conjectured by Guo and Schlosser \cite[Conjectures $12.6$ and $12.7$]{GuoSclssr ConstApprox 21}:
\begin{align*}
\sum_{n=0}^{(p-1)/3}
(6n+1)
\frac{\left(\frac{1}{3}\right)_n^{4}(2n)!}
     {(n!)^{4}\left(\frac{2}{3}\right)_{2n}}
\equiv p \pmod{p^{3}}
\qquad \text{for primes } p\equiv 1 \pmod{3},
\end{align*}
\begin{align*}
\sum_{n=0}^{(p+1)/3}
(6n-1)
\frac{\left(-\frac{1}{3}\right)_n^{4}(2n)!}
     {(n!)^{4}\left(-\frac{2}{3}\right)_{2n}}
\equiv p \pmod{p^{3}}
\qquad \text{for primes } p\equiv 2 \pmod{3},
\end{align*}
where $(z)_n=z(z+1)\cdots(z+n-1)$ denotes the shifted factorial. In this paper, we establish a one-parameter extension of them. The main ingredients of our proofs are the creative microscoping method developed by Guo and Zudilin \cite{guo2019q} and Andrews' multi-series transformation.
\end{abstract}

\maketitle

\section{Introduction and statement of results}

For any complex numbers $z,q$
with $|q|<1$ and a non-negative integer $M$, define the \emph{$q$-shifted factorial} to be
\begin{align*}
(z;q)_{\infty} &:= \prod_{n=0}^{\infty}(1-zq^{n}),\\
(z;q)_{M} &:= \frac{(z;q)_{\infty}}{(zq^{M};q)_{\infty}}.
\end{align*}
For brevity, we also use the abbreviated notation
\begin{align*}
  (z_1,z_2,\ldots,z_N;q)_n := (z_1;q)_n(z_2;q)_n\cdots(z_N;q)_n.  
\end{align*}
Recall that the \emph{$q$-integer} $[m]=[m]_q$ is defined as $[m] :=  \frac{1-q^{m}}{1-q}$. Furthermore, denote by $\Phi_m(q)$, the $m$-th \emph{cyclotomic polynomial}, defined as
\begin{align*}
    \Phi_m(q) := \prod_{\substack{1\leq d\leq m\\ \gcd(d,m)=1}}
\left(q-e^{2\pi i d/m}\right),
\end{align*}
with $i^2=-1$. It is clear that $\Phi_m(q)$ is an
irreducible polynomial in $q$ with integer coefficients and $\Phi_p(q)=[p]$ for any prime $p$. Here and throughout the paper, the polynomials are considered in the ring $\mathbb{Q}[q]$. For any two rational functions $A(q)$ and $B(q)$, we say that $A(q)$ and $B(q)$ in $q$ are congruent modulo a polynomial $C(q)$, written as
\begin{align*}
    A(q)\equiv B(q) \pmod{C(q)},
\end{align*}
if and only if $C(q)$ is a polynomial multiple of the numerator of the reduced form $A(q) - B(q)$.

In his famous letter to Hardy in $1913$, Ramanujan \cite{ramanujan}  mentioned $17$ infinite series representations for $1/\pi$, which take the form
\begin{center}
	$\displaystyle \sum_{n=0}^{\infty} (An + B)\dfrac{C_n}{m^n} = \dfrac{e}{\pi},$
\end{center}
where $A,B,e,m$ are constants and $C_n$ is a product of binomial coefficients. All of these formulas are related to elliptic curves with complex multiplication. In $1997$, Van Hamme \cite{van2020some} developed $p$-adic analogues for $13$ of Ramanujan's infinite series identities and conjectured a list of supercongruences relating certain truncated hypergeometric sums to values of the $p$-adic Gamma function. Since then, supercongruences for truncated hypergeometric series have received considerable attention. We begin by recalling Van Hamme\rq s celebrated $(\mathrm{B}.2)$ supercongruence:
\begin{equation} \label{padic B2}
 \sum_{n=0}^{\frac{p-1}{2}} (-1)^n (4n+1) \frac{(\frac{1}{2})_{n}^3}{(1)_{n}^3} \equiv (-1)^{\frac{p-1}{2}}p \pmod{p^3},
 \end{equation} 
where $p$ is an odd prime and $(z)_n=z(z+1)\cdots(z+n-1)$ denotes the shifted factorial. The supercongruence \eqref{padic B2} was initially confirmed by Mortenson \cite{mortenson} using a ${}_6F_5$ hypergeometric identity. Subsequently, Zudilin \cite{zudilin} reconfirmed the supercongruence \eqref{padic B2} using the WZ-method \cite{wilf}, and later, Long \cite{long1} utilized hypergeometric series identities and evaluations. Sun \cite{Sun2012} obtained an Euler number-related refinement of \eqref{padic B2} modulo $p^4$. There are a few supercongruences related to \eqref{padic B2} in \cite{GuoLiu1}.

In $2021$, Guo and Schlosser \cite[(9.1b)]{GuoSclssr ConstApprox 21} used a ${}_{12}\phi_{11}$ transformation formula to demonstrate that, for any
integer $m>1$ with $m\equiv 1\pmod{3}$,
\begin{align} \label{GS (9.1b)}
\sum_{n=0}^{(m-1)/3}
[6n+1]
\frac{(q;q^3)_n^4(q^3;q^3)_{2n}}
     {(q^3;q^3)_n^4(q^2;q^3)_{2n}}
\equiv 0 \pmod{\Phi_m(q)}.
\end{align}

They also demonstrated, as shown in \cite[(9.3b)]{GuoSclssr ConstApprox 21}, that for any positive integer
$m\equiv 2\pmod{3}$,
\begin{align} \label{GS (9.3b)}
\sum_{n=0}^{(m+1)/3}
[6n-1]
\frac{(q^{-1};q^3)_n^4(q^3;q^3)_{2n}}
     {(q^3;q^3)_n^4(q^{-2};q^3)_{2n}}
\equiv 0 \pmod{\Phi_m(q)}.
\end{align}
In the same paper, they posed the following two conjectural supercongruences related to \eqref{GS (9.1b)} and \eqref{GS (9.3b)} in the cases of $m=p$ and $q \to 1$, modulo higher powers.
\begin{conjecture} [Conjecture $12.6$, \cite{GuoSclssr ConstApprox 21}] \label{GS conj 1}
For $p\equiv 1\pmod{3}$, we have
\begin{align*}
    \sum_{n=0}^{(p-1)/3}
(6n+1)
\frac{(\frac{1}{3})_n^4(2n)!}
     {(n!)^4 (\frac{2}{3})_{2n}}
\equiv p \pmod{p^3}.
\end{align*}
\end{conjecture}

\begin{conjecture} [Conjecture $12.7$, \cite{GuoSclssr ConstApprox 21}] \label{GS conj 2}
For $p\equiv 2\pmod{3}$, we have
\begin{align*}
    \sum_{n=0}^{(p+1)/3}
(6n-1)
\frac{(-\frac{1}{3})_n^4(2n)!}
     {(n!)^4(-\frac{2}{3})_{2n}}
\equiv p \pmod{p^3}.
\end{align*}
\end{conjecture}
Recently, Jana and Kalita \cite{JK RJ22} confirmed the above two Conjectures \ref{GS conj 1} and \ref{GS conj 2}, by deriving a generalized supercongruence relation for truncated hypergeometric series. It is noteworthy that Guo and Schlosser \cite{GuoSclssr ConstApprox 21} remarked that they were unable to find complete $q$-analogues of these two conjectures. 
\par Very recently, Guo and Guo \cite{GuoGuo RM} obtained the following curious $q$-analogues of the above Conjectures \ref{GS conj 1} and \ref{GS conj 2}: for any integer $m>1$ with $m \equiv 1 \pmod{3}$,
\begin{align} \label{Curious QA 1}
    &\sum_{n=0}^{(m-1)/3}
[6n+1]
\frac{(q;q^3)_n^4(q^3;q^3)_{2n}}
     {(q^3;q^3)_n^4(q^2;q^3)_{2n}} \notag \\
&\quad\equiv
(-q)^{1-m}[m]_{q^2}
\sum_{n=0}^{(m-1)/3}
\frac{
(1+q)(-q^3;q^3)_n
(q^{1+m};q^3)_n(q^{1-m};q^3)_n
}{
(1+q^{3n+1})(q^5;q^6)_n(q^3;q^3)_n
}
q^{3n}
\pmod{\Phi_m(q)^3},
\end{align}
and, for any positive integer $m \equiv 2 \pmod{3}$,
\begin{align} \label{Curious QA 2}
    &\sum_{n=0}^{(m+1)/3}
[6n-1]
\frac{(q^{-1};q^3)_n^4(q^3;q^3)_{2n}}
     {(q^3;q^3)_n^4(q^{-2};q^3)_{2n}} \notag \\
&\quad\equiv
(-q)^{-1-m}[m]_{q^2}
\sum_{n=0}^{(m+1)/3}
\frac{
(1+q)(-q^3;q^3)_n
(q^{-1+m};q^3)_n(q^{-1-m};q^3)_n
}{
(1+q^{3n-1})(q;q^6)_n(q^3;q^3)_n
}
q^{3n}
\pmod{\Phi_m(q)^3},
\end{align}
which are extensions of the $q$-congruences \eqref{GS (9.1b)} and \eqref{GS (9.3b)}. 
\par In $2022,$ Guo \cite{guo2022new} extended the Van Hamme\rq s $\mathrm{(A.2)}$ supercongruence with an additional parameter $s$ by providing a $q$-analogue with the help of creative microscoping method designed by Guo and Zudilin \cite{guo2019q}. Over the past few years, parametric extensions of Ramanujan-type supercongruences have attracted many researchers and have led to numerous related results. For example, see \cite{Guo25, Guo-qB2C2, Liton2025 Ramanujan, KJ 2026 RJ, P2 Liton, Tang G2, Tang2023 JMAA, Wang2024 Results Math, Wei2024 RACSAM}.

\par In this paper, we present a parametric generalization of the $q$-congruences \eqref{Curious QA 1} and \eqref{Curious QA 2} involving the additional parameter $s$. Our main results are stated in the following theorems.

\begin{theorem} \label{Main Thm 1}
    Let $m>1$ be an integer satisfying $m \equiv 1 \pmod{3}$, and $s\le (m-1)/3$ a non-negative integer. Then
\begin{align} \label{Main Thm 1 eqn}
&\sum_{n=s}^{(m-1)/3}[6n+1] \frac{(q;q^3)_{n-s}(q;q^3)_{n+s} (q;q^3)_n^2 (q^3;q^3)_{2n}}{(q^3;q^3)_{n-s}(q^3;q^3)_{n+s} (q^3;q^3)_n^2 (q^2;q^3)_{2n}} \notag \\
&\quad \equiv (-1)^{\frac{m-1}{3}-s} q^{1-m(s+1) + 3s^2 + 2s} [m+3s] \frac{1+q^m}{1+q^{3s+1}} \frac{(q^m,q^{m+1},q^{m+2-3s};q^3)_s}{(q^{m+3};q^3)_s (q^2;q^3)_{2s}} \notag \\
&\quad \quad \times \sum_{k=0}^{(m-1)/3-s} \sum_{j=0}^{(m-1)/3-s-k} \left\{ \frac{(q^{3s};q^3)_k (-q;q^3)_j (q^{6s+3};q^6)_k}{(q^3;q^3)_k (q^3;q^3)_j (q^{6s+3},q^{3s+1};q^3)_k} \right. \notag \\ & \qquad \qquad \qquad \qquad \qquad  \times \left. \frac{(-q^{3s+3},q^{3s+1+m}, q^{3s+1-m};q^3)_{k+j}}{(q^{6s+5};q^6)_{k+j} (-q^{3s+4}; q^3)_{k+j}}  \right\} (-q)^k q^{3(k+j)} \pmod{\Phi_m(q)^3}.
\end{align}
\end{theorem}
Clearly, the $s=0$ case of \eqref{Main Thm 1 eqn} is the identity of Guo and Guo\rq s $q$-congruence \eqref{Curious QA 1}.

\begin{theorem} \label{Main Thm 2}
    Let $m$ be a positive integer satisfying $m \equiv 2 \pmod{3}$, and $s\le (m-3)/3$ a non-negative integer. Then
\begin{align} \label{Main Thm 2 eqn}
&\sum_{n=s}^{(m+1)/3}[6n-1] \frac{(q^{-1};q^3)_{n-s}(q^{-1};q^3)_{n+s} (q^{-1};q^3)_n^2 (q^3;q^3)_{2n}}{(q^3;q^3)_{n-s}(q^3;q^3)_{n+s} (q^3;q^3)_n^2 (q^{-2};q^3)_{2n}} \notag \\
&\quad \equiv (-1)^{\frac{m+1}{3}-s} q^{-1-m(s+1) + 3s^2 - 2s} [m+3s] \frac{1+q^m}{1+q^{3s-1}} \frac{(q^m,q^{m-1},q^{m+4-3s};q^3)_s}{(q^{m+3};q^3)_s (q^{-2};q^3)_{2s}} \notag \\
&\quad \quad \times \sum_{k=0}^{(m+1)/3-s} \sum_{j=0}^{(m+1)/3-s-k} \left\{ \frac{(q^{3s};q^3)_k (-q^{-1};q^3)_j (q^{6s+3};q^6)_k}{(q^3;q^3)_k (q^3;q^3)_j (q^{6s+3},q^{3s-1};q^3)_k} \right. \notag \\ & \qquad \qquad \qquad \qquad \qquad  \times \left. \frac{(-q^{3s+3},q^{3s-1+m}, q^{3s-1-m};q^3)_{k+j}}{(q^{6s+1};q^6)_{k+j} (-q^{3s+2}; q^3)_{k+j}}  \right\} (-q^{-1})^k q^{3(k+j)} \pmod{\Phi_m(q)^3}.
\end{align}
\end{theorem}
When $s=0$, the $q$-congruence \eqref{Main Thm 2 eqn} is the identity in \eqref{Curious QA 2}.

\par The rest of the paper is arranged as follows. In Section \ref{Sec2}, we first recall Andrews' multi-series generalization of Watson's $_{8}\phi_7$ transformation formula, and, secondly, we deduce two auxiliary lemmas required for the proofs of our main results. The next two Sections \ref{Sec3} and \ref{Sec4} are devoted to the proofs of Theorems \ref{Main Thm 1} and \ref{Main Thm 2}. Our proofs use the creative microscoping method, recently designed by Guo and Zudilin \cite{guo2019q}, together with Andrews' multi-series transformation.

\section{Preliminaries} \label{Sec2}
In this section, first, we recall the definition of basic hypergeometric series, whose summation and transformation formulas are widely utilized in the proofs of $q$-congruences. Following Gasper and Rahman \cite{gasper}, the basic hypergeometric ${}_{r}\phi_s$ series with $r$ upper parameters $x_1, \ldots, x_r$, $s$ lower parameters $y_1, \ldots, y_s$, base $q$, and argument $z$ are defined by
\begin{align*}
  {}_{r}\phi_{s}
\left[
\begin{matrix}
x_1, x_2, \ldots, x_r \\
y_1, \ldots, y_s
\end{matrix}
; q, z
\right]
:= \sum_{n=0}^{\infty}
\frac{(x_1, x_2, \ldots, x_r; q)_n}
{(q, y_1, \ldots, y_s; q)_n}
\left[(-1)^n q^{\binom{n}{2}}\right]^{1+s-r}
z^n.  
\end{align*}
If one of the upper parameters, $x_r$, is of the form $q^{-m}$, where $m$ is a non-negative integer, then such a series terminates. The series converges for $|z| < 1$ if it does not terminate.

In this paper, we will utilize a remarkable transformation formula attributed to Andrews \cite[Theorem $4$]{AndrewsM}, which can be expressed as follows:

\begin{align} \label{Andrews multi-series}
&\sum_{n\ge 0}
\frac{
(a,q\sqrt{a},-q\sqrt{a},e_1,f_1,\ldots,e_h,f_h,q^{-M};q)_n
}{
(q,\sqrt{a},-\sqrt{a},aq/e_1,aq/f_1,\ldots,
aq/e_h,aq/f_h,aq^{M+1};q)_n
}
\left(
\frac{a^h q^{h+M}}{e_1f_1\cdots e_h f_h}
\right)^n \notag \\
&\quad =
\frac{(aq,aq/e_h f_h;q)_M}
     {(aq/e_h,aq/f_h;q)_M}
\sum_{j_1,\ldots,j_{h-1}\ge 0}
\frac{
(aq/e_1 f_1;q)_{j_1}\cdots
(aq/e_{h-1}f_{h-1};q)_{j_{h-1}}
}{
(q;q)_{j_1}\cdots(q;q)_{j_{h-1}}
}
\notag \\
&\qquad \times
\frac{
(e_2,f_2;q)_{j_1}\cdots
(e_h,f_h;q)_{j_1+\cdots+j_{h-1}}
}{
(aq/e_1,aq/f_1;q)_{j_1}\cdots
(aq/e_{h-1},aq/f_{h-1};q)_{j_1+\cdots+j_{h-1}}
}
\notag \\
&\qquad \times
\frac{
(q^{-M};q)_{j_1+\cdots+j_{h-1}}
}{
(e_h f_h q^{-M}/a;q)_{j_1+\cdots+j_{h-1}}
}
\frac{
(aq)^{j_{h-2}+\cdots+(h-2)j_1}
q^{j_1+\cdots+j_{h-1}}
}{
(e_2 f_2)^{j_1}\cdots
(e_{h-1} f_{h-1})^{j_1+\cdots+j_{h-2}}
}.
\end{align}
This transformation can be seen as a multi-series generalization of Watson's ${}_{8}\phi_7$ transformation formula (see \cite[Appendix (III.18)]{gasper}), which we present below in standard notation for basic hypergeometric series \cite[Section $1$]{gasper} :

\begin{align*} 
&{}_8\phi_7\left[
\begin{matrix}
a, & qa^{\frac{1}{2}}, & -qa^{\frac{1}{2}}, & b, & c, & d, & e, & q^{-M} \\
& a^{\frac{1}{2}}, & -a^{\frac{1}{2}}, & aq/b, & aq/c, & aq/d, & aq/e, & aq^{M+1}
\end{matrix}
~~; q, \dfrac{a^2q^{M+2}}{bcde}
\right] \notag \\
&\qquad =
\frac{(aq,aq/de;q)_M}{(aq/d,aq/e;q)_M}{}_4\phi_3\!\left[
\begin{matrix}
aq/bc, & d, & e, & q^{-M}\\
& aq/b, & aq/c, & deq^{-M}/a
\end{matrix}
;\,q,\;q
\right].
\end{align*}

We now state a crucial lemma due to Guo and Schlosser.
\begin{lemma} [Lemma $2.1$, \cite{GuoSchlsrLemma RM}]  \label{Lemma 1}
    Let $d$, $J$ and $m$ be positive integers with $J\le m-1$, and $\ell $ an integer satisfying
$dJ\equiv -\ell \pmod m$. Then, for $0\le n\le J$, we have
\begin{align} \label{lemma 1 Eqn}
    \frac{(aq^{\ell};q^d)_{J-n}}{(q^d/a;q^d)_{J-n}}
\equiv
(-a)^{J-2n} q^{J(dJ-d+2 \ell)/2+(d-\ell)n}
\frac{(aq^{\ell};q^d)_n}{(q^d/a;q^d)_n}
\pmod{\Phi_m(q)}.
\end{align}
\end{lemma}
With the help of the above Lemma \ref{Lemma 1}, we deduce the following two $q$-congruences modulo $\Phi_m(q)$, which play an important role in our proofs.
\begin{lemma} \label{lemma 2}
    Let $m>1$ be an integer satisfying $m \equiv 1 \pmod{3}$, and $s\le (m-1)/3$ a non-negative integer. Then
\begin{align} \label{lemma 2 eqn}
&\sum_{n=s}^{(m-1)/3}[6n+1] \frac{(q;q^3)_{n-s}(q;q^3)_{n+s} (aq;q^3)_n (q/a;q^3)_n (q^3;q^3)_{2n}}{(q^3;q^3)_{n-s}(q^3;q^3)_{n+s} (aq^3;q^3)_n (q^3/a;q^3)_n (q^2;q^3)_{2n}} \equiv 0 \pmod{\Phi_m(q)}.
\end{align}
\end{lemma}
\begin{proof}
    Setting $\ell =1, d=3$, and $J= \frac{m-1}{3}$ in \eqref{lemma 1 Eqn}, it follows that, for $s \le n \le \frac{m-1}{3}-s$,
    \begin{align}  \label{lemma 2 eqn 2}
&\displaystyle \frac{(aq;q^3)_{\frac{m-1}{3}-n-s}}
{(q^3/a;q^3)_{\frac{m-1}{3}-n+s}} \notag \\
&=
\frac{ \frac{(aq;q^3)_{\frac{m-1}{3}-n-s}}{(q^3/a;q^3)_{\frac{m-1}{3}-n-s}}          }
{(1- \frac{q^{2+m-3n-3s}}{a})
 (1- \frac{q^{5+m-3n-3s}}{a})
 \cdots
 (1- \frac{q^{m-1-3n+3s}}{a})} \notag \\
&\equiv
\frac{
(-a)^{\frac{m-1}{3}-2n-2s }
q^{(m-1)(m-2)/6+2(n+s)} \frac{(aq;q^3)_{n+s}}{(q^3/a;q^3)_{n+s}} }
{(1- \frac{q^{2-3n-3s}}{a})
 (1- \frac{q^{5-3n-3s}}{a})
 \cdots
 (1- \frac{q^{-1-3n+3s}}{a})} \notag \\
&\equiv
(-a)^{\frac{m-1}{3}-2n}
q^{(m-1)(m-2)/6+ 2n + s + 6ns}
\frac{(aq;q^3)_{n-s}}
{(q^3/a;q^3)_{n+s}}
\pmod{\Phi_m(q)},
\end{align} 
where we have used that $q^m \equiv 1 \pmod{\Phi_m(q)}.$ Similarly,
\begin{align} \label{lemma 2 eqn 3}
    \frac{(aq;q^3)_{\frac{m-1}{3}-n+s}}{(q^3/a;q^3)_{\frac{m-1}{3}-n-s}}
\equiv
(-a)^{\frac{m-1}{3}-2n} q^{(m-1)(m-2)/6+ 2n - s - 6ns}
\frac{(aq;q^3)_{n+s}}{(q^3/a;q^3)_{n-s}}
\pmod{\Phi_m(q)}.
\end{align}
Note that
\begin{align*}
    \frac{(q^3;q^3)_{2n}}{(q^2;q^3)_{2n}} = \frac{(q^\frac{3}{2};q^3)_n (-q^{\frac{3}{2}};q^3)_n (q^3;q^3)_n (-q^3;q^3)_n}{(q;q^3)_n (-q;q^3)_n (q^{\frac{5}{2}};q^3)_n (-q^{\frac{5}{2}};q^3)_n}.
\end{align*}
In view of the $q$-congruences \eqref{lemma 1 Eqn}, \eqref{lemma 2 eqn 2}, and \eqref{lemma 2 eqn 3}, it is easy to verify that, for $J= \frac{m-1}{3}$ and $s \le n \le J - s$, 
\begin{align*}
    & [6(J-n)+1] \frac{(q;q^3)_{J-n-s}(q;q^3)_{J-n+s} (aq;q^3)_{J-n} (q/a;q^3)_{J-n} (q^3;q^3)_{2J-2n}}{(q^3;q^3)_{J-n-s}(q^3;q^3)_{J-n+s} (aq^3;q^3)_{J-n} (q^3/a;q^3)_{J-n} (q^2;q^3)_{2J-2n}} \\
    & \equiv - [6n+1] \frac{(q;q^3)_{n-s}(q;q^3)_{n+s} (aq;q^3)_n (q/a;q^3)_n (q^3;q^3)_{2n}}{(q^3;q^3)_{n-s}(q^3;q^3)_{n+s} (aq^3;q^3)_n (q^3/a;q^3)_n (q^2;q^3)_{2n}} \pmod{\Phi_m(q)}.
\end{align*}
This indicates that
\begin{align*}
    \sum_{n=s}^{(m-1)/3-s}[6n+1] \frac{(q;q^3)_{n-s}(q;q^3)_{n+s} (aq;q^3)_n (q/a;q^3)_n (q^3;q^3)_{2n}}{(q^3;q^3)_{n-s}(q^3;q^3)_{n+s} (aq^3;q^3)_n (q^3/a;q^3)_n (q^2;q^3)_{2n}} \equiv 0 \pmod{\Phi_m(q)}.
\end{align*}
For $ \frac{m-1}{3}-s < n \le \frac{m-1}{3}$, the $q$-shifted factorial $(q;q^3)_{n+s}$ contains the factor of the form $1-q^m$, while the denominator on the left-hand side of \eqref{lemma 2 eqn} is coprime to $\Phi_m(q)$. Thus, each term indexed by $n$ on the left-hand side of \eqref{lemma 2 eqn} is congruent to $0$ modulo $\Phi_m(q)$. This proves that the $q$-congruence \eqref{lemma 2 eqn} holds modulo $\Phi_m(q)$.
\end{proof}

\begin{lemma} \label{lemma 3}
    Let $m$ be a positive integer satisfying $m \equiv 2 \pmod{3}$, and $s\le (m-3)/3$ a non-negative integer. Then
\begin{align} \label{lemma 3 eqn}
&\sum_{n=s}^{(m+1)/3}[6n-1] \frac{(q^{-1};q^3)_{n-s}(q^{-1};q^3)_{n+s} (aq^{-1};q^3)_n (q^{-1}/a;q^3)_n (q^3;q^3)_{2n}}{(q^3;q^3)_{n-s}(q^3;q^3)_{n+s} (aq^3;q^3)_n (q^3/a;q^3)_n (q^{-2};q^3)_{2n}} \equiv 0 \pmod{\Phi_m(q)}.
\end{align}
\end{lemma}
\begin{proof}
     Taking $\ell =-1, d=3$, and $J= \frac{m+1}{3}$ in \eqref{lemma 1 Eqn}, it follows that, for $s \le n \le \frac{m+1}{3}-s$,
    \begin{align} \label{lemma 3 eqn 2}
&\displaystyle \frac{(aq^{-1};q^3)_{\frac{m+1}{3}-n-s}}
{(q^3/a;q^3)_{\frac{m+1}{3}-n+s}} \notag \\
&=
\frac{ \frac{(aq^{-1};q^3)_{\frac{m+1}{3}-n-s}}{(q^3/a;q^3)_{\frac{m+1}{3}-n-s}}          }
{(1- \frac{q^{4+m-3n-3s}}{a})
 (1- \frac{q^{7+m-3n-3s}}{a})
 \cdots
 (1- \frac{q^{m+1-3n+3s}}{a})} \notag \\
&\equiv
\frac{
(-a)^{\frac{m+1}{3}-2n-2s }
q^{(m+1)(m-4)/6+4(n+s)} \frac{(aq^{-1};q^3)_{n+s}}{(q^3/a;q^3)_{n+s}} }
{(1- \frac{q^{4-3n-3s}}{a})
 (1- \frac{q^{7-3n-3s}}{a})
 \cdots
 (1- \frac{q^{1-3n+3s}}{a})} \notag \\
&\equiv
(-a)^{\frac{m+1}{3}-2n}
q^{(m+1)(m-4)/6+ 4n - s + 6ns}
\frac{(aq^{-1};q^3)_{n-s}}
{(q^3/a;q^3)_{n+s}}
\pmod{\Phi_m(q)}.
\end{align} 
Similarly,
\begin{align} \label{lemma 3 eqn 3}
    \frac{(aq^{-1};q^3)_{\frac{m+1}{3}-n+s}}{(q^3/a;q^3)_{\frac{m+1}{3}-n-s}}
\equiv
(-a)^{\frac{m+1}{3}-2n} q^{(m+1)(m-4)/6+ 4n + s - 6ns}
\frac{(aq^{-1};q^3)_{n+s}}{(q^3/a;q^3)_{n-s}}
\pmod{\Phi_m(q)}.
\end{align}
Noting that
\begin{align*}
    \frac{(q^3;q^3)_{2n}}{(q^{-2};q^3)_{2n}} = \frac{(q^\frac{3}{2};q^3)_n (-q^{\frac{3}{2}};q^3)_n (q^3;q^3)_n (-q^3;q^3)_n}{(q^{-1};q^3)_n (-q^{-1};q^3)_n (q^{\frac{1}{2}};q^3)_n (-q^{\frac{1}{2}};q^3)_n}.
\end{align*}
Via the $q$-congruences \eqref{lemma 1 Eqn}, \eqref{lemma 3 eqn 2}, and \eqref{lemma 3 eqn 3}, it is not difficult to check that, for $J=\frac{m+1}{3}$ and $s \le n \le J - s$, 

\begin{align*}
    & [6(J-n)-1] \frac{(q^{-1};q^3)_{J-n-s}(q^{-1};q^3)_{J-n+s} (aq^{-1};q^3)_{J-n} (q^{-1}/a;q^3)_{J-n} (q^3;q^3)_{2J-2n}}{(q^3;q^3)_{J-n-s}(q^3;q^3)_{J-n+s} (aq^3;q^3)_{J-n} (q^3/a;q^3)_{J-n} (q^{-2};q^3)_{2J-2n}} \\
    & \equiv -[6n-1] \frac{(q^{-1};q^3)_{n-s}(q^{-1};q^3)_{n+s} (aq^{-1};q^3)_n (q^{-1}/a;q^3)_n (q^3;q^3)_{2n}}{(q^3;q^3)_{n-s}(q^3;q^3)_{n+s} (aq^3;q^3)_n (q^3/a;q^3)_n (q^{-2};q^3)_{2n}} \pmod{\Phi_m(q)}.
\end{align*}
This implies that 
\begin{align*} 
&\sum_{n=s}^{(m+1)/3-s}[6n-1] \frac{(q^{-1};q^3)_{n-s}(q^{-1};q^3)_{n+s} (aq^{-1};q^3)_n (q^{-1}/a;q^3)_n (q^3;q^3)_{2n}}{(q^3;q^3)_{n-s}(q^3;q^3)_{n+s} (aq^3;q^3)_n (q^3/a;q^3)_n (q^{-2};q^3)_{2n}} \equiv 0 \pmod{\Phi_m(q)}.
\end{align*}
When $ n \in \left( \frac{m+1}{3}-s, \frac{m+1}{3} \right]$, the factor $1-q^m$ appears in the $q$-shifted factorial $(q^{-1};q^3)_{n+s}$. This shows that the summand on the left-hand side of \eqref{lemma 3 eqn} is congruent to $0$ modulo $\Phi_m(q)$. This completes the proof.
\end{proof}

\section{Proof of Theorem \ref{Main Thm 1}} \label{Sec3}
In order to prove Theorem \ref{Main Thm 1}, we first establish their parametric form with an additional parameter $a$ so that the generalized $q$-congruence holds modulo $(1-aq^{m}) (a-q^{m}).$
\begin{theorem} \label{Theorem 3}
    Let $m>1$ be an integer satisfying $m \equiv 1 \pmod{3}$, and $s\le (m-1)/3$ a non-negative integer. Then, modulo $(1-aq^m)(a-q^m),$
\begin{align} \label{Theorem 3 eqn}
&\sum_{n=s}^{(m-1)/3}[6n+1] \frac{(q;q^3)_{n-s}(q;q^3)_{n+s} (aq;q^3)_n (q/a;q^3)_n (q^3;q^3)_{2n}}{(q^3;q^3)_{n-s}(q^3;q^3)_{n+s} (aq^3;q^3)_n (q^3/a;q^3)_n (q^2;q^3)_{2n}} \notag \\
&\quad \equiv (-1)^{\frac{m-1}{3}-s} q^{1-m(s+1) + 3s^2 + 2s} [m+3s] \frac{1+q^m}{1+q^{3s+1}} \frac{(q^m,q^{m+1},q^{m+2-3s};q^3)_s}{(q^{m+3};q^3)_s (q^2;q^3)_{2s}} \notag \\
&\quad \quad \times \sum_{k=0}^{(m-1)/3-s} \sum_{j=0}^{(m-1)/3-s-k} \left\{ \frac{(q^{3s};q^3)_k (-q;q^3)_j (q^{6s+3};q^6)_k}{(q^3;q^3)_k (q^3;q^3)_j (q^{6s+3},q^{3s+1};q^3)_k} \right. \notag \\ & \qquad \qquad \qquad \qquad \qquad \qquad \times \left. \frac{(-q^{3s+3},q^{3s+1+m}, q^{3s+1-m};q^3)_{k+j}}{(q^{6s+5};q^6)_{k+j} (-q^{3s+4}; q^3)_{k+j}}  \right\} (-q)^k q^{3(k+j)}.
\end{align}
\end{theorem}
\begin{proof}
For $a=q^{-m}$ or $a=q^{m}$, the left-hand side of \eqref{Theorem 3 eqn} is equal to
\begin{align} \label{Theorem 3 eqn 2}
    & \sum_{n=s}^{(m-1)/3}[6n+1] \frac{(q;q^3)_{n-s}(q;q^3)_{n+s} (q^{1-m};q^3)_n (q^{1+m};q^3)_n (q^3;q^3)_{2n}}{(q^3;q^3)_{n-s}(q^3;q^3)_{n+s} (q^{3-m};q^3)_n (q^{3+m};q^3)_n (q^2;q^3)_{2n}} \notag \\
    & = \sum_{n=0}^{(m-1)/3-s}[6n+6s+1] \frac{(q;q^3)_{n}(q;q^3)_{n+2s} (q^{1-m};q^3)_{n+s} (q^{1+m};q^3)_{n+s} (q^3;q^3)_{2n+2s}}{(q^3;q^3)_{n}(q^3;q^3)_{n+2s} (q^{3-m};q^3)_{n+s} (q^{3+m};q^3)_{n+s} (q^2;q^3)_{2n+2s}}.
\end{align}
Substituting $h=3$ and then taking $q \to q^3, a=q^{6s+1}, e_1=q, f_1= q^{3s+3}, e_2=q^{3s+\frac{3}{2}}, f_2 =-q^{3s+\frac{3}{2}}, e_3= -q^{3s+3}, f_3= q^{3s+1+m}, M \to \frac{m-1}{3}-s$ in Andrews multi-series transformation \eqref{Andrews multi-series}, the right-hand side of \eqref{Theorem 3 eqn 2} can be rewritten as
\begin{align*}
    & [6s+1] \frac{(q^{1+m},q^{1-m};q^3)_s (q;q^3)_{2s}}{(q^{3+m},q^{3-m};q^3)_s (q^2;q^3)_{2s}} \frac{(q^{6s+4},-q^{-m};q^3)_{\frac{m-1}{3}-s}}{(-q^{3s+1}, q^{3s+3-m};q^3)_{\frac{m-1}{3}-s}}  \\
    &\quad \quad \times \sum_{k=0}^{(m-1)/3-s} \sum_{j=0}^{(m-1)/3-s-k} \left\{ \frac{(q^{3s};q^3)_k (-q;q^3)_j (q^{6s+3};q^6)_k}{(q^3;q^3)_k (q^3;q^3)_j (q^{6s+3},q^{3s+1};q^3)_k} \right. \notag \\ & \qquad \qquad \qquad \qquad \qquad \qquad \times \left. \frac{(-q^{3s+3},q^{3s+1+m}, q^{3s+1-m};q^3)_{k+j}}{(q^{6s+5};q^6)_{k+j} (-q^{3s+4}; q^3)_{k+j}}  \right\} (-q)^k q^{3(k+j)} \\
    & = [m+3s] \frac{(q;q^3)_{\frac{m-1}{3}+s} (q^{1+m},q^{1-m};q^3)_s (-q^{-m};q^3)_{\frac{m-1}{3}-s}}{(q^{3+m};q^3)_s (q^{3-m};q^3)_{\frac{m-1}{3}} (q^2;q^3)_{2s} (-q^{3s+1};q^3)_{\frac{m-1}{3}-s}} \\
    &\quad \quad \times \sum_{k=0}^{(m-1)/3-s} \sum_{j=0}^{(m-1)/3-s-k} \left\{ \frac{(q^{3s};q^3)_k (-q;q^3)_j (q^{6s+3};q^6)_k}{(q^3;q^3)_k (q^3;q^3)_j (q^{6s+3},q^{3s+1};q^3)_k} \right. \notag \\ & \qquad \qquad \qquad \qquad \qquad \qquad \times \left. \frac{(-q^{3s+3},q^{3s+1+m}, q^{3s+1-m};q^3)_{k+j}}{(q^{6s+5};q^6)_{k+j} (-q^{3s+4}; q^3)_{k+j}}  \right\} (-q)^k q^{3(k+j)} \\
    & = (-1)^{\frac{m-1}{3}-s} q^{1-m(s+1) + 3s^2 + 2s} [m+3s] \frac{1+q^m}{1+q^{3s+1}} \frac{(q^m,q^{m+1},q^{m+2-3s};q^3)_s}{(q^{m+3};q^3)_s (q^2;q^3)_{2s}}  \\
&\quad \quad \times \sum_{k=0}^{(m-1)/3-s} \sum_{j=0}^{(m-1)/3-s-k} \left\{ \frac{(q^{3s};q^3)_k (-q;q^3)_j (q^{6s+3};q^6)_k}{(q^3;q^3)_k (q^3;q^3)_j (q^{6s+3},q^{3s+1};q^3)_k} \right. \notag \\ & \qquad \qquad \qquad \qquad \qquad \qquad \times \left. \frac{(-q^{3s+3},q^{3s+1+m}, q^{3s+1-m};q^3)_{k+j}}{(q^{6s+5};q^6)_{k+j} (-q^{3s+4}; q^3)_{k+j}}  \right\} (-q)^k q^{3(k+j)}.
\end{align*}
This implies that the above $q$-congruence \eqref{Theorem 3 eqn} is true modulo $1-aq^m$ and $a-q^m$. We complete the proof of Theorem \ref{Theorem 3} because the polynomials $(1-aq^m)$ and $(a-q^m)$ are pairwise coprime.
\end{proof}
\begin{proof}[Proof of Theorem \ref{Main Thm 1}]
    For $1 \le s \le \frac{m-1}{3}$, the $q$-shifted factorial $(q^m;q^3)_s$ has the factor of the form $1-q^m$, whereas the $s=0$ case implies $ \displaystyle [m+3s] \frac{1+q^m}{1+q^{3s+1}}=[m]_{q^2}$, which is a multiple of $1-q^m$. This implies that the right-hand side of \eqref{Theorem 3 eqn} is congruent to $0$ modulo $\Phi_m(q)$. Using Lemma \ref{lemma 2}, we conclude that the $q$-congruence \eqref{Theorem 3 eqn} is true modulo $\Phi_m(q)$.
    Moreover, since the polynomials $\Phi_m(q), (1-aq^m),$ and $(a-q^m)$ are coprime with each other. Thus, by the Chinese remainder theorem for coprime polynomials, the $q$-supercongruence \eqref{Theorem 3 eqn} is true modulo $\Phi_m(q) (1-aq^m)(a-q^m)$. By setting $a=1$ in \eqref{Theorem 3 eqn}, we complete the proof of Theorem \ref{Main Thm 1}.
\end{proof}

\section{Proof of Theorem \ref{Main Thm 2}} \label{Sec4}
The proof of Theorem \ref{Main Thm 2} is similar to that of Theorem \ref{Main Thm 1}. We first deduce the generalized version of the $q$-congruence \eqref{Main Thm 2 eqn} as follows.
\begin{theorem} \label{Theorem 4}
    Let $m$ be a positive integer satisfying $m \equiv 2 \pmod{3}$, and $s\le (m-3)/3$ a non-negative integer. Then, modulo $(1-aq^m)(a-q^m),$
\begin{align} \label{Theorem 4 eqn}
&\sum_{n=s}^{(m+1)/3}[6n-1] \frac{(q^{-1};q^3)_{n-s}(q^{-1};q^3)_{n+s} (aq^{-1};q^3)_n (q^{-1}/a;q^3)_n (q^3;q^3)_{2n}}{(q^3;q^3)_{n-s}(q^3;q^3)_{n+s} (aq^3;q^3)_n (q^3/a;q^3)_n (q^{-2};q^3)_{2n}} \notag \\
&\quad \equiv (-1)^{\frac{m+1}{3}-s} q^{-1-m(s+1) + 3s^2 - 2s} [m+3s] \frac{1+q^m}{1+q^{3s-1}} \frac{(q^m,q^{m-1},q^{m+4-3s};q^3)_s}{(q^{m+3};q^3)_s (q^{-2};q^3)_{2s}} \notag \\
&\quad \quad \times \sum_{k=0}^{(m+1)/3-s} \sum_{j=0}^{(m+1)/3-s-k} \left\{ \frac{(q^{3s};q^3)_k (-q^{-1};q^3)_j (q^{6s+3};q^6)_k}{(q^3;q^3)_k (q^3;q^3)_j (q^{6s+3},q^{3s-1};q^3)_k} \right. \notag \\ & \qquad \qquad \qquad \qquad \qquad \qquad \times \left. \frac{(-q^{3s+3},q^{3s-1+m}, q^{3s-1-m};q^3)_{k+j}}{(q^{6s+1};q^6)_{k+j} (-q^{3s+2}; q^3)_{k+j}}  \right\} (-q^{-1})^k q^{3(k+j)}.
\end{align}
\end{theorem}
\begin{proof}
    For $a=q^{-m}$ or $a=q^{m}$, the left-hand side of \eqref{Theorem 4 eqn} is equal to
\begin{align} \label{Theorem 4 eqn 2}
    & \sum_{n=s}^{(m+1)/3}[6n-1] \frac{(q^{-1};q^3)_{n-s}(q^{-1};q^3)_{n+s} (q^{-1-m};q^3)_n (q^{-1+m};q^3)_n (q^3;q^3)_{2n}}{(q^3;q^3)_{n-s}(q^3;q^3)_{n+s} (q^{3-m};q^3)_n (q^{3+m};q^3)_n (q^{-2};q^3)_{2n}} \notag \\
    & = \sum_{n=0}^{(m+1)/3-s}[6n+6s-1] \frac{(q^{-1};q^3)_{n}(q^{-1};q^3)_{n+2s} (q^{-1-m};q^3)_{n+s} (q^{-1+m};q^3)_{n+s} (q^3;q^3)_{2n+2s}}{(q^3;q^3)_{n}(q^3;q^3)_{n+2s} (q^{3-m};q^3)_{n+s} (q^{3+m};q^3)_{n+s} (q^{-2};q^3)_{2n+2s}}.
\end{align}
Putting $h=3$ and then letting $q \to q^3, a=q^{6s-1}, e_1=q^{-1}, f_1= q^{3s+3}, e_2=q^{3s+\frac{3}{2}}, f_2 =-q^{3s+\frac{3}{2}}, e_3= -q^{3s+3}, f_3= q^{3s-1+m}, M \to \frac{m+1}{3}-s$ in Andrews multi-series transformation \eqref{Andrews multi-series}, the right-hand side of \eqref{Theorem 4 eqn 2} can be rewritten as
\begin{align*}
    & [6s-1] \frac{(q^{-1-m},q^{-1+m};q^3)_s (q^{-1};q^3)_{2s}}{(q^{3-m},q^{3+m};q^3)_s (q^{-2};q^3)_{2s}} \frac{(q^{6s+2},-q^{-m};q^3)_{\frac{m+1}{3}-s}}{(-q^{3s-1}, q^{3s+3-m};q^3)_{\frac{m+1}{3}-s}}  \\
    &\quad \quad \times \sum_{k=0}^{(m+1)/3-s} \sum_{j=0}^{(m+1)/3-s-k} \left\{ \frac{(q^{3s};q^3)_k (-q^{-1};q^3)_j (q^{6s+3};q^6)_k}{(q^3;q^3)_k (q^3;q^3)_j (q^{6s+3},q^{3s-1};q^3)_k} \right. \notag \\ & \qquad \qquad \qquad \qquad \qquad \qquad \times \left. \frac{(-q^{3s+3},q^{3s-1+m}, q^{3s-1-m};q^3)_{k+j}}{(q^{6s+1};q^6)_{k+j} (-q^{3s+2}; q^3)_{k+j}}  \right\} (-q^{-1})^k q^{3(k+j)} \\
    & = [m+3s] \frac{(q^{-1};q^3)_{\frac{m+1}{3}+s} (q^{-1-m},q^{-1+m};q^3)_s (-q^{-m};q^3)_{\frac{m+1}{3}-s}}{(q^{3+m};q^3)_s (q^{3-m};q^3)_{\frac{m+1}{3}} (q^{-2};q^3)_{2s} (-q^{3s-1};q^3)_{\frac{m+1}{3}-s}} \\
    &\quad \quad \times \sum_{k=0}^{(m+1)/3-s} \sum_{j=0}^{(m+1)/3-s-k} \left\{ \frac{(q^{3s};q^3)_k (-q^{-1};q^3)_j (q^{6s+3};q^6)_k}{(q^3;q^3)_k (q^3;q^3)_j (q^{6s+3},q^{3s-1};q^3)_k} \right. \notag \\ & \qquad \qquad \qquad \qquad \qquad \qquad \times \left. \frac{(-q^{3s+3},q^{3s-1+m}, q^{3s-1-m};q^3)_{k+j}}{(q^{6s+1};q^6)_{k+j} (-q^{3s+2}; q^3)_{k+j}}  \right\} (-q^{-1})^k q^{3(k+j)} \\
    & = (-1)^{\frac{m+1}{3}-s} q^{-1-m(s+1) + 3s^2 - 2s} [m+3s] \frac{1+q^m}{1+q^{3s-1}} \frac{(q^m,q^{m-1},q^{m+4-3s};q^3)_s}{(q^{m+3};q^3)_s (q^{-2};q^3)_{2s}}  \\
&\quad \quad \times \sum_{k=0}^{(m+1)/3-s} \sum_{j=0}^{(m+1)/3-s-k} \left\{ \frac{(q^{3s};q^3)_k (-q^{-1};q^3)_j (q^{6s+3};q^6)_k}{(q^3;q^3)_k (q^3;q^3)_j (q^{6s+3},q^{3s-1};q^3)_k} \right. \notag \\ & \qquad \qquad \qquad \qquad \qquad \qquad \times \left. \frac{(-q^{3s+3},q^{3s-1+m}, q^{3s-1-m};q^3)_{k+j}}{(q^{6s+1};q^6)_{k+j} (-q^{3s+2}; q^3)_{k+j}}  \right\} (-q^{-1})^k q^{3(k+j)}.
\end{align*}
This shows that the above $q$-congruence \eqref{Theorem 4 eqn} is true modulo $1-aq^m$ and $a-q^m$. Since the polynomials $(1-aq^m)$ and $(a-q^m)$ are pairwise coprime. Thus, the $q$-congruence \eqref{Theorem 4 eqn} holds modulo $(1-aq^m)(a-q^m)$. This completes the proof.
\end{proof}
\begin{proof}[Proof of Theorem \ref{Main Thm 2}]
    For $ 0 \le s \le \frac{m-3}{3}$, it is easy to see that the right-hand side of \eqref{Theorem 4 eqn} is congruent to $0$ modulo $\Phi_m(q)$. Applying the Chinese remainder theorem for coprime polynomials, we obtain the $q$-congruence \eqref{Theorem 4 eqn} modulo $\Phi_m(q)(1-aq^m)(a-q^m)$. Hence, taking $a=1$ in this $q$-congruence yields the desired $q$-congruence \eqref{Main Thm 2 eqn}. This immediately completes the proof of Theorem \ref{Main Thm 2}.
\end{proof}

\section*{Declarations}
{\bf Funding:} No funding was received for this research.\\
{\bf Conflicts of Interest:} The author has no conflicts
of interest. \\
{\bf Data Availability Statement:} No datasets were generated or analyzed in this study. The results are obtained entirely through theoretical and mathematical methods.


\begin{thebibliography}{99}

\bibitem{AndrewsM}  G. E. Andrews, Problems and prospects for basic hypergeometric functions, in Theory and Application for Basic Hypergeometric Functions, Publication of the Mathematics Research Center, University of Wisconsin, Vol. $35$, Academic Press, New York, $1975$, pp. $191–224$.

\bibitem{gasper}
G. Gasper and M. Rahman, Basic hypergeometric series,
Encyclopedia of Mathematics and its Applications, vol. $96$,
2nd ed., Cambridge University Press, Cambridge, $2004$.

\bibitem{guo2022new}
V. J. W. Guo, A new extension of the (A.2) supercongruence
of Van Hamme, Results Math. \textbf{77} $(2022)$, no. $3$, Art. $96$.

\bibitem{Guo25}
V. J. W. Guo, A new $q$-variation of the (C.2) supercongruence
of Van Hamme, Results Math. \textbf{79} $(2024)$, Art. $87$, $13$ pp.

\bibitem{Guo-qB2C2}
V. J. W. Guo, A further common $q$-generalization of the
(B.2) and (C.2) supercongruences of Van Hamme,
Int. J. Number Theory \textbf{21} $(2025)$, $605-614$.

\bibitem{GuoGuo RM}  S. Guo and V.J.W. Guo, Curious $q$-analogues of two supercongrunces modulo the third power of a prime, Results Math. \textbf{76}, $37 ~(2021)$.

\bibitem{GuoLiu1}
V. J. W. Guo and J. C. Liu, Some congruences related to a congruence of Van Hamme, Integral Transforms Spec. Funct. \textbf{31} $(2020)$, no. $3, 221-231$.

\bibitem{GuoSchlsrLemma RM} V.J.W. Guo and M.J. Schlosser, A new family of $q$-supercongruences modulo the fourth power of a cyclotomic polynomial. Results Math. $75$, $155~ (2020)$.

\bibitem{GuoSclssr ConstApprox 21}
V. J. W. Guo and M. J. Schlosser, Some $q$-supercongruences from
transformation formulas for basic hypergeometric series,
Constr. Approx. \textbf{53} $(2021)$, $155-200$.


\bibitem{guo2019q}
V. J. W. Guo and W. Zudilin, A $q$-microscope for supercongruences,
Adv. Math. \textbf{346} $(2019)$, $329-358$.

\bibitem{JK RJ22} A. Jana and G. Kalita, Proof of some conjectural supercongruences of Guo and Schlosser, Ramanujan J. \textbf{54}, $649–658 ~(2021)$.

\bibitem{Liton2025 Ramanujan}
A. Jana and L. Karmakar, Generalization of some supercongruences
related to (B.2) and (C.2) supercongruences of Van Hamme,
Ramanujan J. \textbf{67} $(2025)$, Art. $54$, $13$ pp.

\bibitem{KJ 2026 RJ}
L. Karmakar and A. Jana, A common $q$-generalization of the
(F.2) and (G.2) supercongruences of Van Hamme,
Ramanujan J. \textbf{71} $(2026)$, Art. $27$, $10$ pp.

\bibitem{P2 Liton}
S.-X. Li, L. Karmakar, A. Jana, and S.-D. Wang,
Some $q$-supercongruences for sums involving certain
$q$-shifted factorials, submitted.


\bibitem{long1}
L. Long, Hypergeometric evaluation identities and supercongruences,
Pacific J. Math. \textbf{249} $(2011)$, no. 2, $405-418$.

\bibitem{mortenson}
E. Mortenson, A $p$-adic supercongruence conjecture of Van Hamme,
Proc. Amer. Math. Soc. \textbf{136} $(2008)$, no. $12, 4321-4328$.

\bibitem{wilf}
M. Petkov$\hat{s}$ek, H. S. Wilf, and D. Zeilberger, $A=B$,
A. K. Peters, Wellesley, MA, $1996$.

\bibitem{ramanujan}
S. Ramanujan, Modular equations and approximations to $\pi$,
Quart. J. Math. \textbf{45} $(1914)$, $350-372$.
Reprinted in Collected Papers of Srinivasa Ramanujan,
AMS Chelsea Publ., Providence, RI, $2000$, pp. $23-39$.

\bibitem{Sun2012} Z.-W. Sun, A refinement of a congruence result by van Hamme and Mortenson, Ill. J. Math. \textbf{56}, $967–979 ~(2012)$.

\bibitem{Tang G2}
N. Tang, New generalisations of Van Hamme's (G.2) supercongruence,
Bull. Aust. Math. Soc. \textbf{107} $(2023)$, $177-183$.

\bibitem{Tang2023 JMAA}
N. Tang, Some $q$-supercongruences related to Van Hamme's
(C.2) supercongruence, J. Math. Anal. Appl. \textbf{527} $(2023)$,
Art. $127402$.

\bibitem{van2020some}
L. Van Hamme, Some conjectures concerning partial sums of generalized
hypergeometric series, in $p$-adic Functional Analysis
(Nijmegen, $1996$), Lecture Notes in Pure and Appl. Math., vol. $192$,
Marcel Dekker, New York, $1997$, pp. $223-236$.


\bibitem{Wang2024 Results Math}
C. Wang, A new extension of Van Hamme's (E.2) supercongruence,
Results Math. \textbf{79} $(2024)$, Art. $97$.

\bibitem{Wei2024 RACSAM}
C. Wei, Generalizations of the (F.2) supercongruence of Van Hamme,
Rev. R. Acad. Cienc. Exactas Fis. Nat. Ser. A Mat.
\textbf{118} $(2024)$, Art. $33$.

\bibitem{zudilin}
W. Zudilin, Ramanujan-type supercongruences,
J. Number Theory \textbf{129} $(2009)$, no. $8, 1848-1857$.


\end{thebibliography}
\end{document}